\documentclass[reqno,10pt]{amsart}
\usepackage{amsmath, amsfonts,amsthm,amssymb,amsbsy,upref,color,graphicx,amscd,hyperref,enumerate}
\usepackage[active]{srcltx}
\usepackage[latin1]{inputenc}
\usepackage{bbm}
\usepackage{mathrsfs} 
\usepackage{xfrac}

\def\ds{\displaystyle}
\def\eps{{\varepsilon}}
\def\N{\mathbb{N}}

\def\Om{\Omega}

\def\R{\mathbb{R}}

\def\HH{\mathcal{H}}

\def\Dr{D}

\def\reg{Reg(\partial\Omega_U)}
\def\sing{Sing_1(\partial\Omega_U)}
\def\2{Sing_2(\partial\Omega_U)}

\newcommand{\be}{\begin{equation}}
\newcommand{\ee}{\end{equation}}

\numberwithin{equation}{section}
\theoremstyle{plain}
\newtheorem{teo}{Theorem}[section]
\newtheorem{lemma}[teo]{Lemma}

\newtheorem{prop}[teo]{Proposition}
\newtheorem{deff}[teo]{Definition}
\newtheorem{oss}[teo]{Remark}

\newcommand{\mean}[1]{\,-\hskip-1.08em\int_{#1}} 
\newcommand{\ind}{\mathbbm{1}}

\title[Regularity for the vectorial Bernoulli problem]{Regularity of the free boundary for the vectorial Bernoulli problem}

\author{Dario Mazzoleni, Susanna Terracini, Bozhidar Velichkov}

\address{Dario Mazzoleni: \newline \indent
Dipartimento di Matematica e Fisica ``N. Tartaglia'',
Universit\`a Cattolica,\newline \indent
Via dei Musei, 41,
25121 Brescia, Italy,} 
\email{dariocesare.mazzoleni@unicatt.it}

\address{Susanna Terracini: \newline \indent
 Dipartimento di Matematica ``Giuseppe Peano'', 
Universit\`a di Torino, \newline \indent
Via Carlo Alberto, 10,
10123 Torino, Italy,} 
\email{susanna.terracini@unito.it}

\address {Bozhidar Velichkov: \newline \indent
Laboratoire Jean Kuntzmann, Universit\'e Grenoble Alpes
\newline \indent
B\^atiment IMAG, 700 Avenue Centrale, 38401 Saint-Martin-d'H\`eres, France 
}
\email{bozhidar.velichkov@univ-grenoble-alpes.fr}

\date{\today}

\begin{document}

\begin{abstract}
In this paper we study the regularity of the free boundary for a vector-valued Bernoulli problem, with no sign assumptions on the boundary data. More precisely, given an open, smooth set of finite measure $D\subset \R^d$, $\Lambda>0$ and $\varphi_i\in H^{\sfrac12}(\partial D)$, we deal with 
\[
\min{\left\{\sum_{i=1}^k\int_D|\nabla v_i|^2+\Lambda\Big|\bigcup_{i=1}^k\{v_i\not=0\}\Big|\;:\;v_i=\varphi_i\;\mbox{on }\partial D\right\}}.
\]
We prove that, for any optimal vector $U=(u_1,\dots, u_k)$, the free boundary $\partial (\cup_{i=1}^k\{u_i\not=0\})\cap D$ is made of a regular part, which is relatively open and locally the graph of a $C^\infty$ function, a (one-phase) singular part, of Hausdorff dimension at most $d-d^*$, for a $d^*\in\{5,6,7\}$, and by a set of branching (two-phase) points, which is relatively closed and of finite $\HH^{d-1}$ measure.
Our arguments are based on the NTA structure of the regular part of the free boundary.
\end{abstract}

\thanks{{\bf Acknowledgments.} 
D.~Mazzoleni and S.~Terracini are partially  supported ERC Advanced Grant 2013 n. 339958
{\it Complex Patterns for Strongly Interacting Dynamical Systems - COMPAT}, by the PRIN-2012-74FYK7 Grant {\it Variational and perturbative aspects of nonlinear differential problems}. B.Velichkov has been partially supported by the LabEx PERSYVAL-Lab (ANR-11-LABX-0025-01) project GeoSpec and the project ANR CoMeDiC}

\maketitle


%

\section{Introduction} 
Free boundary problems arise in models describing several physical phenomena, as for example thermal insulation, and have been an important topic of mathematical study in the last four decades starting from the seminal work~\cite{altcaf}. The huge literature on this topics has provided many new tools, which have been employed also in very different fields.
In two recent papers~\cite{csy,mtv} the authors consider a vector-valued Bernoulli problem, under the assumption that at least one of the components does not change sign. In this paper we give an answer to the main open question from \cite{csy,mtv}, proving the regularity of the free boundary without any assumption on the sign of the components. Our main result is that in a neighborhood of a flat point (that is, a point of Lebesgue density $\sfrac12$) at least one of the components has constant sign. Our analysis strongly relies on the approach and the results from~\cite{mtv}.

\medskip

Given a smooth open set $D\subset \R^d$, $\Lambda>0$ and $\Phi=(\varphi_1,\dots,\varphi_k)\in H^{\sfrac12}(\partial D; \R^k)$, that is $\varphi_i\in H^{\sfrac12}(\partial D)$, for $i=1,\dots,k$, we consider the vectorial free boundary problem  
\begin{equation}\label{vectfb}
\min{\left\{\int_D|\nabla U|^2\,dx+\Lambda\left|\Omega_U\right|\;:\;U\in H^{1}(D; \R^k),\  U=\Phi\;\mbox{on }\partial D\right\}},
\end{equation}
where, for a vector-valued function $U=(u_1,\dots,u_k):D\to\R^k$, we use the notations 
$$|U|:=\sqrt{u_1^2+\dots+u_k^2}\ ,\qquad |\nabla U|^2:=\sum_{i=1}^k|\nabla u_i|^2\qquad\text{and}\qquad \Omega_U:=\{|U|>0\}=\bigcup_{i=1}^k\{u_i\not=0\}\subset D.$$
We will refer to the set $\partial\Omega_U\cap D$ as to the \emph{free boundary} given by $U$.
Our main result is the following:
\begin{teo}\label{regvectfb}
There exists a solution to problem~\eqref{vectfb}.
Any solution $U\in H^1(D;\R^k)$ is Lipschitz continuous in $D\subset\R^d$ and the set $\Omega_U$ has a locally finite perimeter in $D$. The free boundary $\partial\Omega_U\cap D$ is a disjoint union of a regular part $\reg$, a (one-phase) singular set $\sing$ and a set of branching points $\2$. 
\begin{enumerate}[(1)]
\item The regular part $\reg$ is an open subset of $\partial\Omega_U$ and is locally the graph of a $C^\infty$ function.
\item The one-phase singular set $\sing$ consists only of points in which the Lebesgue density of $\Omega_U$ is strictly between $\sfrac12$ and $1$.
Moreover, there is $d^*\in\{5,6,7\}$ such that: 
\begin{itemize}
\item if $d<d^*$, then $\sing$ is empty;
\item if $d=d^*$, then the singular set $\sing$ contains at most a finite number of isolated points;
\item if $d>d^*$, then the $(d-d^\ast)$-dimensional Hausdorff measure of $\sing$ is locally finite in $D$.
\end{itemize}
\item The set of branching points $\2$ is a closed set 
of locally finite $(d-1)$-Hausdorff measure in $D$
and consists only of points in which the Lebesgue density of $\Omega_U$ is $1$ and the blow-up limits are linear functions.
\end{enumerate}
\end{teo}

\subsection{Remarks on the one-phase singular set $\sing$}
The critical dimension $d^*$ is the lowest dimension at which the free boundaries of the one-phase scalar Alt-Caffarelli problem (see \cite{altcaf}) admit singularities. Caffarelli, Jerison and Kënig proved in \cite{cjk} that $d^\ast\ge 4$, Jerison and Savin \cite{js} showed that $d^\ast\ge 5$, while De Silva and Jerison \cite{dsj} gave an example of a singular minimal cone in dimension $7$, so $d^*\in\{5,6,7\}$. The first claim of Theorem \ref{regvectfb} (2) follows by the fact that at points of the one-phase singular set $\sing$ the blow-up limits of the minimizers of \eqref{vectfb} are multiples of a solution of the one-phase scalar Alt-Caffarelli problem (Subsection \ref{sub:BUclassification}). The second claim of Theorem \ref{regvectfb} (2) was proved in~\cite[Section~5.5]{mtv} together with the Hausdorff dimension bound 
$$\dim_{\mathcal H}(\sing)\le d-d^\ast,\quad\text{for}\quad d>d^*,$$
which follows by a dimension reduction argument based on the Weiss' monotonicity formula. The claim of Theorem \ref{regvectfb} (2) was proved by Edelen and Engelstein in ~\cite[Theorem~1.15]{ee} by a finer argument based on the quantitative dimension reduction of Naber and Valtorta \cite{nv1,nv2}. We notice that \cite{ee} contains also a stratification result on $\sing$.  

\subsection{Further results on the set of branching points $\2$}
Under the assumption that one of the components of the optimal vector has constant sign \cite{csy,mtv}, all one-homogeneous singular solutions are multiples of global solutions for the one-phase scalar problem. In this case, the singular set of $\partial\Omega_U$ is given precisely by $\sing$. Without the constant sign assumption, the structure of the singular set changes drastically. A set $\2$ of branching points, in which the free boundary may form cusps pointing inwards, might appear. This is natural since the scalar case corresponds to the two-phase Bernoulli problem, for which this is a well-known, though non completely understood, phenomenon. In particular, the dimension of this set of branching points can be as big as the dimension of the regular free boundary. This is somehow natural since for the two-phase case the branching points are contact points of the two level sets $\{u>0\}$ and $\{u<0\}$. 

The free boundaries around branching points for the vectorial problem have more complex structure. Indeed, even in dimension two, true cusps may appear on the free boundary that is, around a branching point $x_0$, the set $B_r(x_0)\cap \Omega_U$ might stay connected, while the Lebesgue density $\frac{|B_r\setminus\Omega_U|}{|B_r|}$ might decay as $r$ goes to zero (see \cite{sv} for an example of such a free boundary).  On the other hand, the nodal set may also degenerate into linear subspace of codimension higher than one (see Lemma \ref{l:lin_sol} for an example of homogeneou solution with a thin nodal set). In Section \ref{s:branch}, using a Federer Reduction Principle, we prove a stratification result, Theorem \ref{t:stratification}, for the branching points, which in particular shows that  the only significant (in terms of Hausdorff measure) set of branching points is the one for which the nodal set degenerates into a $d-1$ dimensional plane. 

\subsection{Relation with shape optimization problems for the eigenvalues of the Dirichlet Laplacian}
The vectorial Bernoulli problem is strictly related to a whole class of shape optimization problems involving the eigenvalues of the Dirichet Laplacian. In particular, suppose that $U^\ast=(u_1^\ast,\dots,u_k^\ast)$ is the vector whose components are the Dirichlet eigenfunctions on the set $\Omega^\ast$, solution of the shape optimization problem 
$$\min\Big\{\sum_{j=1}^k\lambda_j(\Omega)\ :\ \Omega\subset\R^d,\ \Omega\ \text{open},\ |\Omega|=1\Big\}.$$
It was proved in \cite{mtv} that $U^\ast$ is a quasi-minimizer of \eqref{vectfb}. Thus, the regularity of the optimal set $\Omega^\ast$ is strongly related to (not to say a consequence of) the regularity of the free boundaries of the solutions of \eqref{vectfb}. A result for more general functionals was proved by Kriventsov and Lin~\cite{kl}, still under some structural assumption on the free boundary. It was then extended by the same authors to general spectral functionals in ~\cite{kl1}. The shape optimization problem considered in ~\cite{kl1} corresponds to \eqref{vectfb} with sign changing components. On the other hand the nature of the spectral functionals forces the authors to take a very different road and use an approximation with functionals for which the constant sign assumption is automatically satisfied. In particular, they do select a special representative of the optimal set, which roughly speaking corresponds to the biggest {\it quasi-open} set which solves the problem.  The problem \eqref{vectfb} allows a more direct approach and in particular our regularity result holds for the free boundary of \emph{any} optimal vector. 

\subsection{Plan of the paper and sketch of the proof of Theorem \ref{regvectfb}}

Since the existence of an optimal vector is nowadays standard, we start Subsection \ref{sub:lip} by proving the Lipschitz continuity of $U$, which follows by the fact that each component is quasi-minimizer for the scalar Alt-Caffarelli functional and so, by \cite{bmpv}, is Lipschitz continuous. 
In Subsection \ref{finper} we prove that the positivity set $\Omega_U$ has finite perimeter in $D$ and that the $(d-1)$-Hausdorff measure of $\partial\Omega_U$ is finite. Our argument is different from the classical approach of Alt and Caffarelli and is based on a comparison of the energy of the different level sets of $|U|$. 

In Subsection \ref{sub:blowupexistence} we summarize the convergence results on the blow-up sequences and Subsection \ref{sub:BUclassification} is dedicated to the classification of the blow-up limits, which are one-homogeneous global minimizers (that is, globally defined local minimizers) of \eqref{vectfb} (see Remark \ref{r:hm}). In Lemma \ref{l:lin_sol} we show that a new class of global minimizers appears with respect to the problem considered in \cite{csy,mtv}. In Lemma \ref{l:bw} we classify the possible blow-up limits according to the Lebesgue density; this is the main result of the section. Finally, in Definition \ref{def:reg_sing}, we define the sets $\reg$, $\sing$ and $\2$. 

\medskip

In Section \ref{s:reg} we prove the smoothness of $\reg$. In Subsection \ref{sub:visc} we prove that on the one-phase free boundary $\reg\cup\sing$, $U$ satisfies the extremality condition $|\nabla|U||=\sqrt{\Lambda}$ in a viscosity sense.  In Subsection \ref{sub:NTA} we prove that $\reg$ is Reifenberg flat and NTA domain.

Subsection \ref{sub:main} deal with the proof that in a neighborhood of a point $x_0\in\reg$ at least one of the components of $U$ remains strictly positive and (up to a multiplicative constant) controls $|U|$ (see Lemma \ref{costsign}). This is the main result of this Section and the proof is based on the geometric properties of NTA domains and on the Boundary Harnack Principle. 
 In Subsection \ref{sub:C1alpha} and Subsection \ref{sub:Cinfty} we prove that $\reg$ is respectively $C^{1,\alpha}$ and $C^\infty$. The result of  Lemma \ref{costsign} allows us to apply the results from \cite{mtv}. We give the main steps of the proof for the sake of completeness. 

\medskip

Section \ref{s:branch} is dedicated to the study of the set $\2$ of points $x_0\in\partial\Omega_U$ in which all the blow-up limits $U_0\in \mathcal{BU}_U(x_0)$ are linear functions of the form $U_0(x)=Ax$. In Subsection \ref{sub:rank} we prove that the rank of the linear map $U_0$ depends only on the point $x_0$ and we define the $j$th stratum $\mathcal S_j$ as the set of points for which this rank is precisely $j$. In Subsection \ref{sub:dim} we use a dimension reduction argument in the spirit of Federer to prove that the Hausdorff dimension of each stratum $\mathcal S_j$ is $d-j$. Finally, in Subsection \ref{sub:crit} we give a criterion for the uniqueness of the blow-up limits in terms of the Lebesgue density of $\Omega_U$.

\section{Boundary behavior of the solutions}
The existence of an optimal vector $U=(u_1,\dots, u_k)$ is standard and follows by the direct method of the calculus of variations (for more details we refer to \cite{altcaf}). 

\subsection{Lipschitz continuity and non-degeneracy}\label{sub:lip}
{\it Any minimizer $U$ has the following properties:
\begin{enumerate}[(i)]
\item The vector-valued function $U:D\to\R^k$ is locally Lipschitz continuous in $D$. 
\item The real-valued function $|U|$ is non-degenerate, i.e. there are constants $c_0>0$ and $r_0>0$ such that for every $x_0\in \partial\Omega_U\cap D$ and $r\in(0,r_0]$ we have
\begin{equation}\label{e:nondeg}
\Big(\mean{B_r(x_0)}{|U|\,dx}<c_0 r\Big)\Rightarrow \Big(U\equiv 0\ \ \text{in}\ \ B_{r/2}(x_0)\Big).
\end{equation}
\item There are constants $\eps_0,\;r_0$ such that the \emph{lower density estimate} holds: 
\begin{equation}\label{densestbel}
\eps_0|B_r|\le \big|\Omega_U\cap B_r(x_0)\big|,\quad\text{for every}\quad x_0\in \partial\Omega_U\cap D\quad\text{and}\quad r\le r_0.
\end{equation}
\end{enumerate}}

\begin{oss}
\rm Claim {(i)} in particular implies that, {\it for every} minimizer $U$ of \eqref{vectfb}, the set $\Omega_U$ is open. 
\end{oss}

\begin{oss}
\rm It is important to highlight that, unlike the case treated in~\cite{mtv,csy} where it was assumed at least one component $u_i$ to be positive, we cannot hope to have a density estimate \emph{from above} on $\partial\Om_U\cap D$. Actually, we expect a set of branching points (cusps) will come out. Indeed, the case $k=1$ corresponds to a scalar {\it two-phase} problem for which (at least in dimension two) the set $\Omega_U$ is composed of two $C^{1,\alpha}$ sets (see \cite{sv}). At the points of the common boundary of these two sets, the Lebesgue density of $\Omega_U$ is $1$. 
\end{oss}

\begin{proof}[Proof of (i)]
The Lipschitz continuity of each component $u_i$, $i=1,\dots,k$, descends from a quasi-minimality property. Indeed, reasoning as in~\cite[Section~6.2]{mtv}, for every $\tilde u_i:D\to\R$ such that $\tilde u_i-u_i\in H^1_0(D)$ we consider the competitor $\tilde U:=(u_1,\dots,\tilde u_i,\dots,u_k)$. By the optimality of $U$ we have 
$$\int_D |\nabla u_i|^2\,dx+\Lambda\big|\{|U|>0\}\big|\le \int_D |\nabla \tilde u_i|^2\,dx+\Lambda\big|\{|\tilde U|>0\}\big|, $$ 
which implies that each component $u_i$ is a quasi-minimizer of the Dirichlet energy, that is
\begin{equation}\label{quasimindiren}
\int |\nabla u_i|^2\,dx\le \int |\nabla \tilde u_i|^2\,dx+\Lambda |B_r|\quad\text{for every}\quad \tilde u_i\quad \text{such that}\quad \tilde u_i-u_i\in H^1_0(B_r).
\end{equation}
Applying \cite[Theorem~3.3]{bmpv} we get that $u_i$ is Lipschitz continuous in $\Dr$, and since $i=1,\dots,k$ is arbitrary, so is $U$. This concludes the proof of {(i)}.\end{proof} 

\begin{proof}[Proof of (ii) and (iii)]
The non-degeneracy of $|U|$ follows by~\cite[Lemma~2.6]{mtv}, which can be applied since $U$ satisfies the condition~(2.9) therein with $K=0$ and every $\eps>0$. Finally, we notice that the density estimate \emph{from below}  \eqref{densestbel} holds for every Lipschitz function satisfying the non-degeneracy condition \eqref{e:nondeg} (see for example ~\cite[Lemma~2.11]{mtv} or \cite{altcaf}).
\end{proof}

\subsection{Finiteness of the perimeter}\label{finper} {\it For any optimal vector $U\in H^1(D;\R^k)$, solution of~\eqref{vectfb}, the set $\Om_U$ has locally finite perimeter in $D$ and, moreover, }
\begin{equation}\label{e:Hd-1}
\HH^{d-1}(\partial\Omega_U\cap K)<\infty\quad\text{for every compact set}\quad K\subset D.
\end{equation}
\begin{oss}
We notice that the condition \eqref{e:Hd-1} is more general than the finiteness of the perimeter since $\partial\Omega_U$ may contain points $x_0$ which are in the measure theoretic interior of $\Omega_U$ that is, $|B_r(x_0)\setminus\Omega_U|=0$.
\end{oss}
\noindent In order to prove the claim of this Subsection, we will use the following lemma, which holds in general.
\begin{lemma}\label{l:finite_perimeter_lemma}
Suppose that $D\subset\R^d$ is an open set and that $\phi:D\to[0,+\infty]$ is a function in $H^1(D)$ for which there exist $\overline \eps>0$ and $C>0$ such that
\begin{equation}\label{epsest}
\int_{\{0\leq \phi\leq \eps\}\cap D}|\nabla \phi|^2\,dx+\Lambda\big|\{0\leq \phi\leq \eps\}\cap D\big|\leq C \eps\,,\quad\text{for every}\quad 0<\eps\leq \overline\eps.
\end{equation}
Then $P(\{\phi>0\};D)\le C\sqrt\Lambda$.  
\end{lemma}
\begin{proof}
By the co-area formula, the Cauchy-Schwarz inequality and \eqref{epsest}, we have that, for every $\eps\le\overline\eps$, 
\begin{align*}
\int_0^\eps \mathcal H^{d-1}\big(\{\phi=t\}\cap D\big)\,dt&=\int_{\{0\leq \phi\leq \eps\}}|\nabla \phi|\,dx\leq \Big(\int_{\{0\leq \phi\leq \eps\}}|\nabla\phi|^2\,dx\Big)^{\sfrac12}\big|\{0\leq \phi\leq \eps\}\big|^{\sfrac12}\leq \eps C\sqrt\Lambda.
\end{align*}
Taking $\eps=1/n$, we get that there is $\delta_n\in[0,1/n]$ such that 
$$\mathcal H^{d-1}\big(\partial^*\{\phi>\delta_n\}\cap D\big)\leq n\int_0^{1/n} \mathcal H^{d-1}\big(\{\phi=t\}\cap D\big)\,dt\le C\sqrt\Lambda.$$
Passing to the limit as $n\to\infty$, we obtain $\mathcal H^{d-1}\big(\partial^*\{\phi>0\}\cap D\big)\leq C\sqrt\Lambda$, which concludes the proof of the lemma.
\end{proof}

\begin{proof}[Proof of the claim of Subsection \ref{finper}.]
We aim to prove an estimate of the form \eqref{epsest} for $\phi=|U|$ by constructing a suitable competitor.
Since we want to prove a local result, we take $x_0\in \partial \Om\cap D$ and $B_r(x_0)\subset D$; moreover we can assume without loss of generality that $x_0=0$ and $r=1$.
Setting $\rho:=|U|$, for every $\eps>0$, we define 
\[
\begin{split}
\widetilde U=(\tilde u_1,\dots,\tilde u_k)&:=\frac{(\rho-\eps)_{+}}{\rho}U\ ,\qquad\text{where}\qquad \widetilde u_i=\left(1-\frac{\eps}{\rho}\right)_{+}u_i\quad\text{for every}\quad i=1,\dots, k,
\end{split}
\]
and, for a smooth function $\phi\in C^{\infty}(D)$ such that $0\le \phi\le 1$ in $D$, $\phi=1$ in $B_{\sfrac12}$ and $\phi=0$ on $D\setminus B_1$,
\[
V=(v_1,\dots, v_k):=(1-\phi)U+\phi\widetilde U=\left\{
\begin{split}
(1-\phi)U,\quad \mbox{if }|U|=\rho<\eps,\\
\left(1-\eps\frac{\phi}{\rho}\right)U,\quad \mbox{if }|U|=\rho\geq \eps.
\end{split}
\right.
\]
Thus, clearly $V$ is an admissible competitor in problem~\eqref{vectfb}.

We observe that the following relations, which we will use in the rest of the proof, hold true:\[
|\nabla \rho\,|\leq |\nabla U|\quad\text{and}\quad \sum_{i=1}^ku_i\nabla u_i=\rho\nabla \rho\quad\text{in}\quad D\,;\quad \frac{\eps\phi}{\rho}\leq 1\quad\mbox{in}\quad \{|U|\geq \eps\}.
\]
We can now compute on $\{|U|\geq \eps\}$ 
\[
\begin{split}
|\nabla V|^2-|\nabla U|^2&=\sum_{i}\Big|\nabla \Big(1-\frac{\eps\phi}{\rho}\Big)u_i\Big|^2-|\nabla u_i|^2\\
&= \Big(-\frac{2\eps\phi}{\rho}+\frac{\eps^2\phi^2}{\rho^2}\Big)|\nabla U|^2+\rho^2\Big|\nabla\frac{\eps\phi}{\rho}\Big|^2-2(\rho-\eps\phi)\nabla\rho\cdot \nabla\frac{\eps\phi}{\rho}\\
&= \eps^2|\nabla \phi|^2-2\eps\nabla \phi\cdot\nabla \rho+\Big(|\nabla \rho|^2-|\nabla U|^2\Big)\Big(2\eps\frac{\phi}{\rho}-\eps^2\frac{\phi^2}{\rho^2}\Big)\leq \eps^2|\nabla \phi|^2-2\eps\nabla \phi\cdot\nabla \rho\leq C_1\,\eps,
\end{split}
\]
where $C_1$ depends only on $\|\nabla\phi\|_{L^\infty}$ and $\|\nabla U\|_{L^\infty}$.
Next, on the set $\{|U|<\eps\},$ we compute
\[
\begin{split}
|\nabla U|^2-|\nabla V|^2&=|\nabla U|^2-|\nabla(1-\phi)U|^2\\
&=(2\phi-\phi^2)|\nabla U|^2+2(1-\phi)U\nabla \phi\cdot\nabla U+|U|^2|\nabla \phi|^2\geq |\nabla U|^2\ind_{B_{\sfrac12}}-C_2\,\eps, 
\end{split}
\]
where again $C_2$  depends only on $\|\nabla\phi\|_{L^\infty}$ and $\|\nabla U\|_{L^\infty}$.
By testing the optimality of $U$ with $V$ we get
\[
\int_{B_1}|\nabla U|^2+\Lambda\big|\{0\leq |U|\leq \eps\}\cap B_1\big|\leq \int_{B_1}|\nabla V|^2+\Lambda\big|\{|V|>0\}\cap B_1\big|,
\]
so we deduce\[
\int_{\{0\leq |U|\leq \eps\}}\Big(|\nabla U|^2-|\nabla V|^2\Big)+\Lambda\big|\{0\leq |U|\leq \eps\}\cap B_{\sfrac12}\big|\leq \int_{\{|U|\geq \eps\}}\Big(|\nabla V|^2-|\nabla U|^2\Big)\leq C_1\eps,
\]
and finally, since $V=0$ on the set $\{0\leq |U|\leq \eps\}\cap\{\phi=1\}$, we get 
\[
\qquad\qquad\int_{\{0\leq |U|\leq \eps\}\cap B_{\sfrac12}}|\nabla U|^2+\Lambda\big|\{0\leq |U|\leq \eps\}\cap B_{\sfrac12}\big|\leq (C_1+C_2)\eps,\qquad\qquad
\]
and, since $|\nabla\rho|\le |\nabla U|$ we obtain the estimate \eqref{epsest} for $\rho=|U|$ in the ball $B_{\sfrac12}$. This proves that $\Omega_U$ has locally finite perimeter in $D$. In order to prove \eqref{e:Hd-1} we notice that Lemma \ref{l:finite_perimeter_lemma} gives the following stronger result: 
{\it There is a sequence $\eps_n\to0$ such that the set $\Omega_n:=\{|U|>\eps_n\}$ is such that $\HH^{d-1}(\partial\Omega_n\cap K)<C$ for some universal constant $C$.} 
In particular, for every $n$ we have that there is a cover $\{B_{\eps_n}(x_i)\}_i$ of $\partial\Omega_n\cap K$ such that
$$C\ge \HH^{d-1}(\partial\Omega_n\cap K)\ge C_d\sum_i \eps_n^{d-1}.$$ 
Now, by the non-degeneracy of $U$ there is another universal constant $C$ such that   the family of balls $\{B_{C\eps_n}(x_i)\}_i$ is a cover also for $\partial \Omega_U\cap K$. Since $n$ is arbitrary and the constants are universal, we get the claim.
\end{proof}

\subsection{Compactness and convergence of the blow-up sequences}\label{sub:blowupexistence}

Let $U:D\to\R^k$ be a solution of \eqref{vectfb} or, more generally, a Lipschitz function. For $r\in(0,1)$ and $x\in\R^d$ such that $U(x)=0$, we define 
$$\ds U_{r,x}(y):=\frac1r U(x+ry).$$
When $x=0$ we will use the notation $U_r:=U_{r,0}$.

Suppose now that $(r_n)_{n\ge0}\subset\R^+$ and $(x_n)_{n\ge0}\subset D$ are two sequences such that
\begin{equation}\label{rnxn}
\lim_{n\to\infty}r_n=0,\qquad \lim_{n\to\infty}x_n=x_0\in D,\qquad B_{r_n}(x_n)\subset D\quad\text{and}\quad x_n\in\partial\{|U|>0\}\quad\text{for every}\quad n\ge0.
\end{equation}
Then the sequence $\{U_{r_n,x_n}\}_{n\in\N}$ is uniformly Lipschitz and locally uniformly bounded in $\R^d$. Thus, up to a subsequence, $U_{r_n,x_n}$ converges, as $n\to\infty$, locally uniformly to a Lipschitz continuous function $U_0:\R^d\to\R^k$. Moreover, if $U$ is a minimizer of~\eqref{vectfb}, then for every $R>0$ the following properties hold (see~\cite[Proposition~4.5]{mtv}): 
{\it\begin{enumerate}[(i)]
\item $\ds U_{r_n,x_n}$ converges to $U_0$ strongly in $H^1(B_R;\R^k)$. 
\item The sequence of characteristic functions $\ind_{\Omega_n}$ converges in $L^1(B_R)$ to $\ind_{\Omega_0}$, where 
$$\Omega_n:=\{|U_{r_n}|>0\}\quad\text{and}\quad \Omega_0:=\{|U_0|>0\}.$$
\item The sequences of closed sets $\overline\Omega_n$ and $\Omega_n^c$ converge Hausdorff in $B_R$ respectively to $\overline\Omega_0$ and $\Omega_0^c$.
\item $U_0$ is non-degenerate at zero, that is, there is a dimensional constant $c_d>0$ such that 
$$\|U_0\|_{L^\infty(B_r)}\ge c_d\,r\quad\text{for every}\quad r>0.$$
\end{enumerate}}

\begin{deff}\rm
Let $U:\R^d\to\R^k$ be a Lipschitz function, $r_n$ and $x_n$ be two sequences satisfying \eqref{rnxn}. 
We say that the sequence $U_{r_n,x_n}$ is a {\it blow-up sequence with variable center} (or a {\it pseudo-blow-up}).  
If the sequence $x_n$ is constant, $x_n=x_0$ for every $n\ge 0$, we say that $U_{r_n,x_0}$ is a {\it blow-up sequence with fixed center}. We denote by $\mathcal{BU}_U(x_0)$ the space of all the limits of blow-up sequences with fixed center $x_0$.
\end{deff}

\subsection{Classification of the blow-up limits}\label{sub:BUclassification} In this section we prove that for any $x_0\in\partial\Omega_U\cap D$ the blow-up limits $U_0\in \mathcal{BU}_U(x_0)$ have one of the following forms: 
\begin{itemize}
\item {\it Multiples of a scalar solution of the one-phase problem}, that is there is a one-homogeneous non-negative global minimizer $u:\R^d\to\R^+$ of the one-phase Alt-Caffarelli functional $$\ds\mathcal F(u)=\int|\nabla u|^2\,dx+\Lambda\big|\{u>0\}\big|,$$ such that
\begin{equation}\label{e:one_phase_bw}
U_0(x)=\xi\, u(x),\quad\text{where}\quad \xi\in\R^k\quad\text{and}\quad |\xi|=1.
\end{equation}
\item {\it Linear functions}, that is there is a matrix $A=(a_{ij})_{ij}\in \mathcal{M}_{d\times k}(\R)$ such that
\begin{equation}\label{e:harmonic_bw}
U_0(x)=Ax.
\end{equation}
\end{itemize}
It was shown in \cite{mtv} that every function of the form \eqref{e:one_phase_bw} is a global solution of \eqref{vectfb}. 
In the following lemma we classify the linear solutions. 
\begin{lemma}\label{l:lin_sol}
Let $u:\R^d\to\R^k$ be a linear function, $u(x)=Ax$ with $A=(a_{ij})_{ij}\in M_{d\times k}(\R)$. If $$\ \|A\|:=\ds\sum_{i=1}^k\sum_{j=1}^d a_{ij}^2\ge \Lambda,$$ 
then $u$ is a solution of \eqref{vectfb} in the unit ball $B_1$. Moreover, if $\text{rank}\,A=1$, then the condition $\|A\|\ge \Lambda$ is also necessary.
\end{lemma}
\begin{proof}
Let us first show that if $\ \ds\|A\|\ge \Lambda$, then $u=:(u_1,\dots,u_k)$ is as solution of \eqref{vectfb}. Let $\tilde u=(\tilde u_1,\dots,\tilde u_k):B_1\to\R^d$ be such that $\tilde u=u$ on $\partial B_1$. We will show that $\tilde u$ has a higher energy than $u$. Notice that each component $u_j$, $j=1,\dots,k$, can be written as $u_j(x)=\alpha_j\,v_j(x)$, where $\alpha_j\in\R$ and $v_j(x)=x\cdot\nu_j$ for some $\nu_j\in\partial B_1$. We will also write $\tilde u_j(x)=\alpha_j \tilde v_j(x)$ and we notice that $\tilde v_j=v_j$ on $\partial B_1$. Now since $(v_j)_+$ and $(v_j)_-$ are solutions of the one-phase scalar Alt-Caffarelli problem we have that 
\begin{align*}
\int_{B_1}|\nabla v_j|^2\,dx+|B_1|&=\int_{B_1}|\nabla (v_j)_+|^2\,dx+|\{v_j>0\}\cap B_1|+\int_{B_1}|\nabla (v_j)_-|^2\,dx+|\{v_j<0\}\cap B_1|\\
&\le\int_{B_1}|\nabla (\tilde v_j)_+|^2\,dx+|\{\tilde v_j>0\}\cap B_1|+\int_{B_1}|\nabla (\tilde v_j)_-|^2\,dx+|\{\tilde v_j<0\}\cap B_1|\\
&\le \int_{B_1}|\nabla \tilde v_j|^2\,dx+|\Omega_{\tilde u}\cap B_1|.
\end{align*}  
Multiplying by $\alpha_j^2$, taking the sum over $j$, and using that $\ds \|A\|=\sum_{j=1}^k\alpha_j^2$, we obtain 
\begin{align*}
\int_{B_1}|\nabla u|^2\,dx+\|A\|\,|B_1|&=\sum_{j=1}^k\alpha_j^2\left(\int_{B_1}|\nabla v_j|^2\,dx+|B_1|\right)\\
&\le \sum_{j=1}^k\alpha_j^2\left(\int_{B_1}|\nabla \tilde v_j|^2\,dx+|\Omega_{\tilde u}\cap B_1|\right)= \int_{B_1}|\nabla \tilde u|^2\,dx+\|A\|\,|\Omega_{\tilde u}\cap B_1|.
\end{align*}  
Now since $\Lambda\le \|A\|$, we have
\begin{align*}
\int_{B_1}|\nabla u|^2\,dx+\Lambda\,|B_1|\le  \int_{B_1}|\nabla \tilde u|^2\,dx+\Lambda\,|\Omega_{\tilde u}\cap B_1|.
\end{align*}  

We will now prove that if $\text{rank}\,A=1$ and $\|A\|< \Lambda$, then $u$ is not a solution of \eqref{vectfb}. Indeed, let $u=(u_1,\dots,u_k)$ be as above: $u_j(x)=\alpha_j\, x\cdot\nu_j$ for some $\nu_j\in\partial B_1$. \end{proof}
The classification of the blow-up limits strongly relies on the monotonicity of the vectorial Weiss' boundary adjusted energy introduced in \cite{mtv} 
\begin{equation}\label{weiss_fun}
W(U,x_0,r):=\frac{1}{r^{d}}\left(\int_{B_r(x_0)}|\nabla U|^2\,dx+\Lambda\big|\{|U|>0\}\cap B_r(x_0)\big|\right)-\frac{1}{r^{d+1}}\int_{\partial B_r(x_0)}|U|^2\,d\HH^{d-1},
\end{equation}
which turns out to be monotone in $r$. Precisely, by \cite[Proposition~3.1]{mtv} we have the following estimate. 
\begin{lemma}[Weiss monotonicity formula]\label{mono_weiss} Let $U=(u_1,\dots,u_k)$ be a minimizer for problem~\eqref{vectfb} and $x_0\in \partial\Omega_U\cap D$. Then, the function $r\mapsto W(U,x_0,r)$ is non-decreasing and 
\begin{equation}\label{e:derivataW}
\frac{d}{dr}W(U,x_0,r)\geq\frac{1}{r^{d+2}}\sum_{i=1}^k\int_{\partial B_r(x_0)} |(x-x_0)\cdot\nabla u_i-u_i|^2\,d\HH^{d-1}(x),
\end{equation}
in particular, the limit $\ds \lim_{r\to 0^+}W(U,x_0,r)$ exists and is finite. 
\end{lemma}
\begin{oss}[Homogeneity and minimality of the blow-up limits]\label{r:hm} \rm
As a consequence of the monotonicity formula, we obtain that if $U$ is a solution of \eqref{vectfb}, $x_0\in\partial\Omega_U\cap D$ and $U_0\in\mathcal{BU}_U(x_0)$, then $U_0$ is a one-homogeneous global solution of the vectorial Bernoulli problem. Precisely, the fact that $U_0$ is a global solution  follows by \cite[Proposition 4.2]{mtv}, while for the homogeneity of $U_0$ we use the fact that $U_0$ is a blow-up limit,  $U_0=\lim_{n\to\infty}U_{r_n,x_0}$, and the scaling property of the Weiss energy
$$W(U,x_0,rs)=W(U_{r,x_0},s,0)\quad\text{for every}\quad r,s>0,$$ 
which gives that the function $s\mapsto W(U_0,s,0)$ is constant. In fact, for every $s>0$, we have
$$W(U_0,s,0)=\lim_{n\to \infty}W(U_{r_n,x_0},s,0)=\lim_{n\to \infty}W(U,r_ns,x_0)=\lim_{r\to 0}W(U,r,x_0).$$
Now,  the homogeneity of $U_0$ follows by  \eqref{e:derivataW} applied to $U_0$ and its components.
\end{oss}
\begin{oss}[Lebesgue and energy density]\label{rem:density} \rm Keeping the notation from Remark \ref{r:hm}, we notice that the homogeneity of the blow-up limits and the strong convergence of the blow-up sequences gives  
$$W(U_0,1,0)=\Lambda\big|\{|U_0|>0\}\cap B_1\big|=\lim_{r\to 0}W(U,r,x_0)=\Lambda\omega_d \,\lim_{r\to 0}\frac{|\Omega_U\cap B_r(x_0)|}{|B_r|},$$
for every $U_0\in \mathcal{BU}_U(x_0)$. That is, the energy density $\lim_{r\to 0}W(U,r,x_0)$ coincides, up to a multiplicative constant, with the Lebesgue density, which (as a consequence) exists in every point $x_0$ of the free boundary. In particular, we get 
$$\Omega_U^{(\gamma)}=\left\{x\in\partial\Omega_U\;:\;\lim_{r\rightarrow 0}\frac{|\Omega_U\cap B_r(x)|}{|B_r|}=\gamma\right\}=\Big\{x_0\in\partial\Omega_U\ :\ \lim_{r\to0}W(U,x_0,r)=\Lambda\omega_d \gamma\Big\}.$$
\end{oss}

\begin{lemma}[Structure of the blow-up limits]\label{l:bw}
Let $U$ be a solution of~\eqref{vectfb}, $x_0\in \partial\Omega_U\cap D$. Then, there is a dimensional constant $0<\delta<\sfrac12$ such that precisely one of the following holds: 
\begin{enumerate}[(i)]
\item The Lebesgue density of $\Omega_U$ at $x_0$ is $\sfrac12$ and every blow-up $U_0\in \mathcal{BU}_U(x_0)$ is of the form 
\begin{equation}\label{e:flat_bw}
U_0(x)=\xi (x\cdot\nu)_+\quad\text{where}\quad\xi\in\R^k,\ |\xi|=\sqrt\Lambda,\ \nu\in\R^d,\  |\nu|=1.
\end{equation}
\item The Lebesgue density of $\Omega_U$ at $x_0$ satisfies 
$$\sfrac12+\delta\le \lim_{r\to0}\frac{|\Omega_U\cap B_r(x_0)|}{|B_r|}\le 1-\delta,$$
and every blow-up in $\mathcal{BU}_U(x_0)$ is a one-phase blow-up of the form \eqref{e:one_phase_bw} with singularity in zero. 
\item The Lebesgue density of $\Omega_U$ at $x_0$ is $1$ and every blow-up  in $\mathcal{BU}_U(x_0)$ is of the form \eqref{e:harmonic_bw}. 
\end{enumerate}
\end{lemma}
\begin{proof} Let $x_0\in\partial\Omega_U\cap D$. \\
{\it Step 1.} The following claim holds true: 
$$x_0\in \Omega_U^{(\sfrac12)}\Leftrightarrow \text{there is $U_0\in \mathcal{BU}_U(x_0)$ of the form \eqref{e:flat_bw}} \Leftrightarrow \text{every $U_0\in \mathcal{BU}_U(x_0)$ is of the form \eqref{e:flat_bw}}.$$
Indeed, if one blow-up is of the form \eqref{e:flat_bw}, then by Remark \ref{rem:density} $x_0\in\Omega^{(\sfrac12)}$. On the other, hand, if $x_0\in \Omega^{(\sfrac12)}$ and $U_0\in \mathcal{BU}_U(x_0)$, then again by Remark \ref{rem:density} $|\Omega_{U_0}\cap B_1|=\frac12|B_1|$. The homogeneity of $U_0$ and the fact that $\Delta U_0=0$ on $\Omega_{U_0}$ imply that each component of $U_0$ is an eigenfunction on the sphere corresponding to the eigenvalue $(d-1)$. By the Faber-Krahn inequality on the sphere we get that, up to a rotation, $\Omega_{U_0}=\{x_d>0\}$ and all the components of $U_0$ are multiples of $x_d^+$, that is $U_0(x)=\xi\,x_d^+$ for some $\xi \in\R^k$. Let $\phi$ be a compactly supported function and let $\widetilde U_0=\xi\, (x_d^++\phi)$. Testing the optimality of $U_0$ against $\widetilde U_0$, it is immediate to check (see \cite{mtv}) that $|\xi|\, x_d^+$ is a global minimizer of the one-phase Alt-Caffarelli functional. Thus, an internal perturbation (see \cite{altcaf}) gives $|\xi|=\sqrt{\Lambda}$ and concludes {\it Step 1}. \\
{\it Step 2.} The following claim holds true: 
$$x_0\in \Omega_U^{(1)}\Leftrightarrow \text{there is $U_0\in \mathcal{BU}_U(x_0)$ is of the form \eqref{e:harmonic_bw}}\Leftrightarrow \text{every $U_0\in \mathcal{BU}_U(x_0)$ is of the form \eqref{e:harmonic_bw}}.$$
Indeed, if one blow-up $U_0\in \mathcal{BU}_U(x_0)$ is of the form \eqref{e:harmonic_bw}, then by Remark \ref{rem:density} $x_0\in \Omega_U^{(1)}$. On the other hand, if $x_0\in \Omega_U^{(1)}$, then still by Remark \ref{rem:density} $|U_0\cap B_1|=|B_1|$ and so, the minimality of $U_0$ implies that $U_0$ is harmonic in $B_1$. Now the homogeneity of $U_0$ implies that it is a linear function, $U_0(x)=Ax$, for some matrix $A=(a_{ij})_{ij}$. \\
{\it Step 3.} Finally, suppose that $x_0\in (\partial\Omega_U\cap D)\setminus (\Omega_U^{(\sfrac12)}\cup\Omega_U^{(1)})$ and let $x_0\in\Omega_U^{(\gamma)}$ for some $\gamma\in (0,\sfrac12)\cup(\sfrac12,1)$. Let $U_0=(u_1,\dots,u_k)\in\mathcal{BU}_U(x_0)$. Then each component $u_i$ is $1$-homogeneous and the functions $u_i^+$ and $u_i^-$ are eigenfunctions corresponding to the eigenvalue $d-1$ on the spherical sets $\{u_i>0\}\cap \partial B_1$ and $\{u_i<0\}\cap \partial B_1$. Now since the density $\gamma<1$, we get that at least one of the sets is empty. Thus, none of the components $u_i$ change sign and they are all multiples of the first eigenfunction on the set $\Omega_{U_0}\cap \partial B_1$, that is $U_0=\xi |U_0|$ for some $\xi\in\R^k$. Now, reasoning as in~\cite[Section~5.2]{mtv}, we get that $|\xi|=\Lambda$ and that $|U_0|$ is a global solution of the one-phase scalar functional $\ \ds u\mapsto\int|\nabla u|^2\,dx+|\{u>0\}|$. In particular, the density estimate for the one-phase Alt-Caffarelli functional implies that $\gamma<1-\delta$ for some dimensional constant $\delta>0$. Now, the fact that the first eigenvalue on $\Omega_{U_0}\cap \partial B_1$ is $(d-1)$ implies that $\gamma\ge \sfrac12$. As in \cite[Section~5.2]{mtv}, the improvement of flatness for the scalar problem now implies that $\gamma>\sfrac12+\delta$, which concludes the proof. 
\end{proof}

\begin{deff}\label{def:reg_sing}
Let $x_0\in\partial\Omega_U$. We say that: 

$\bullet$ $x_0$ is a regular point, $x_0\in \reg$, if (i) holds;

$\bullet$ $x_0$ is a (one-phase) singular point, $x_0\in \sing$, if (ii) holds;

$\bullet$ $x_0$ is a branching point, $x_0\in \2$, if (iii) holds.

\end{deff}
\noindent In view of Lemma \ref{l:bw} we have that 
$$\reg=\Om_U^{(\sfrac12)} \cap D,\quad \2=\Om_U^{(1)}\cap \partial \Om_U\cap D,$$
$$\quad \sing=(\partial \Om_U\cap D)\setminus ( \2 \cup \reg).$$

\begin{lemma}\label{l:open_closed}
$\2$ is a closed set and $\reg$ is an open subset of $\partial\Omega_U$.
\end{lemma}
\begin{proof}
We first notice that the function $W(U,x_0,0):=\lim_{r\to0^+}W(U,x_0,r)$ is upper semi-continuous in $x_0$. This follows by the fact that $(x_0,r)\mapsto W(U,x_0,r)$ is increasing in $r>0$ and continuous in $x_0$. Thus, the first part of the claim follows since in the points $x_0\in\2$ the density $W(U,x_0,0)$ is maximal. The second part of the claim follows by the lower density gap from Lemma \ref{l:bw} (2) and the  argument of \cite[Proposition 5.6]{mtv}. 
\end{proof}
\section{Regularity of the one-phase free boundary} \label{s:reg}
Following the argument from \cite{mtv}, we first deduce the optimality condition on the free boundary in a viscosity sense, then we notice that $\reg$ is open and Reifenberg flat. Next we show that around every point of $\reg$ at least one of the components of the optimal vector $U$ has a constant sign. Thus we fall into the framework of \cite{mtv} and can concude the proof by using the boundary Harnack principle in NTA domains and the regularity of the one-phase free boundaries for the scalar problem.
Finally, thanks to Lemma~\ref{costsign}, we can apply the arguments of~\cite[Section~5]{mtv} in order to obtain the $C^\infty$ regularity of $\reg$, using the component of locally constant sign provided by Lemma~\ref{costsign} instead of $u_1$ in the boundary Harnack principle~\cite[Lemma~5.12]{mtv}. We recall here the updated statements for the reader's sake.
\subsection{The stationarity condition on the free boundary}\label{sub:visc}
It is well-known (see for example \cite{altcaf}) that if $u$ is a local minimizer of the Alt-Caffarelli functional 
$$H^1_{loc}(\R^d)\ni u\mapsto\mathcal F(u)=\int|\nabla u|^2\,dx+\Lambda|\{u>0\}|,$$
and the boundary $\partial\{u>0\}$ is smooth, then 
$|\nabla u|=\sqrt{\Lambda}$ on $\partial \{u>0\}. $
There are various ways to state this optimality for free boundaries that are not a priori smooth (see for example \cite{altcaf}, \cite{desilva} and the references therein). In the case of vector-valued functionals, we use the notion of viscosity solution from ~\cite{mtv}. 

\begin{deff}\label{viscoptimality}\rm
Let $\Omega\subset\R^d$ be an open set. We say that the continuous function $U=(u_1,\dots,u_k):\overline\Omega\to\R^k$ is a viscosity solution of the problem 
$$-\Delta U=0\quad\text{in}\quad\Omega,\qquad U=0\quad\text{on}\quad\partial\Omega\cap D,\qquad |\nabla |U||=\sqrt \Lambda\quad\text{on}\quad\partial\Omega\cap D,$$
if for every $i=1,\dots,k$ the component $u_i$ is a solution of the PDE
$$-\Delta u_i=0\quad\text{in}\quad\Omega,\qquad u_i=0\quad\text{on}\quad\partial\Omega\cap D,$$
and the boundary condition
$|\nabla |U||=\sqrt \Lambda\quad\text{on}\quad\partial\Omega\cap D,$
holds in viscosity sense, that is

$\bullet$ for every continuous $\varphi:\R^d\to\R$, differentiable in $x_0\in\partial\Omega\cap D$ and such that ``$\varphi$ touches $|U|$ from below in $x_0$'' (that is $|U|-\varphi:\overline\Omega\to\R$ has a local minimum equal to zero in $x_0$), we have $|\nabla \varphi|(x_0)\le\sqrt\Lambda$.

$\bullet$ for every function $\varphi:\R^d\to\R$, differentiable in $x_0\in\partial\Omega\cap D$ and such that ``$\varphi$ touches $|U|$ from above in $x_0$'' (that is $|U|-\varphi :\overline\Omega\to\R$ has a local maximum equal to zero in $x_0$), we have $|\nabla \varphi|(x_0)\ge\sqrt\Lambda$.
\end{deff}
\begin{lemma}\label{optvisc}
Let $U$ be a minimizer for~\eqref{vectfb} and $x_0\in\reg\cup\sing$. Then, there is $r>0$ such that $U$ is a viscosity solution of 
\begin{equation}\label{viscsolUtion}
-\Delta U=0\quad\text{in}\quad\Omega_U\cap B_r(x_0),\quad U=0\quad\text{on}\quad\partial\Omega_U\cap B_r(x_0),\quad  |\nabla |U||=\sqrt{\Lambda}\quad\text{on}\quad\partial\Omega_U\cap B_r(x_0).
\end{equation}
\end{lemma}
\begin{proof}
Suppose that $\varphi$ touches $|U|$ from above in $y_0\in B_r(x_0)$. Then $|\varphi(y_0)|\ge \Lambda$ precisely as in~\cite[Lemma~5.2]{mtv}. If $\varphi$ touches $|U|$ from below in $y_0$, then every blow-up $U_0\in \mathcal{BU}_{U}(y_0)$ is a one-homogeneous global minimizer of \eqref{vectfb} such that $\Omega_{U_0}$ contains the half-space $\{x\,:\,\nabla\varphi(y_0)\cdot x<0\}$. Now since the Lebesgue density of $\Omega_{U_0}$ is strictly smaller than one, the argument of~\cite[Lemma~5.2]{mtv} gives that all the components of $U_0$ must be multiples of the same global minimizer of the scalar one-phase Alt-Caffarelli problem. Thus $\Omega_{U_0}=\{x\,:\,\nabla\varphi(y_0)\cdot x<0\}$ and the conclusion follows as in~\cite[Lemma~5.2]{mtv}. 
\end{proof}
\subsection{Reifenberg flat and NTA domains}\label{sub:NTA}
In this section we briefly recall the basic geometric properties of the Reifenberg flat and NTA domains. The Reifenberg flatness of $Reg(\partial\Omega_U)$ follows preciesly as in \cite{mtv}. Then a result by Kenig and Toro~\cite{kt1} shows that it is also NTA. In the next section we will use the NTA property to prove regularity. For more details on the properties and the structure of the Reifenberg flat domains we refer to~\cite{kt1}, while NTA domains were studied in~\cite{kt1,jk}.

\begin{deff}[Reifenberg flat domains]
Let $\Omega\subset\R^d$ be an open set and let $0<\delta<\sfrac12$, $R>0$. We say that $\Omega$ is a $(\delta,R)$-Reifenberg flat domain if:
\begin{enumerate}
\item For every $x\in\partial\Omega$ and every $0<r\le R$ there is a hyperplane $H=H_{x,r}$ containing $x$ such that 
$$\text{dist}_{\mathcal H}(B_r(x)\cap H,B_r(x)\cap\partial\Omega)<r \delta. $$
\item For every $x\in\partial\Omega$, one of the connected components of the open set $B_R(x)\cap\{x\ :\ \text{dist}(x,H_{x,R})>2\delta R\}$ is contained in $\Omega$, while the other one is contained in $\R^d\setminus\overline\Omega$.
\end{enumerate}
\end{deff}

\begin{teo}[Reifenberg flat implies NTA, {\cite[Theorem 3.1]{kt1}}]\label{reifimplnta}
There exists a $\delta_0>0$ such that if $\Om\subset \R^d$ is a $(\delta, R)$-Reifenberg flat 
domain for $\delta <\delta_0$, then it is NTA, that is there exist constants  $M>0$ and $r_0>0$ (called NTA constants) such that 
\begin{enumerate}
\item $\Omega$ satisfies the \emph{corkscrew condition}, that is, given $x\in \partial \Om$ and $r\in(0,r_0)$, there exists $x_0\in\Om$ s.t. \[
M^{-1}r<dist(x_0,\partial \Om)<|x-x_0|<r,
\]
\item $\R^d\setminus \Omega$ satisfies the corkscrew condition,
\item If $w\in \partial \Om$ and $w_1,w_2\in B(w,r_0)\cap \Om$, then there is a rectifiable curve $\gamma\colon [0,1]\rightarrow \Om$ with $\gamma(0)=w_1$ and $\gamma(1)=w_2$ such that
$\HH^1(\gamma([0,1]))\leq M|w_1-w_2|$ and 
$$\min{\{\HH^1(\gamma([0,t])),\HH^1(\gamma([t,1]))\}}\leq M dist(\gamma(t),\partial \Om)\quad \text{for every}\quad t\in[0,1].$$
\end{enumerate} 
\end{teo}

\begin{oss}\label{rem:ntapalleconn}\rm
We note that an NTA domain $\Om\subset \R^d$ is obviously connected, while its intersection with a ball is not necessarily so. This is due to the fact that an arc, contained in $\Omega$ and connecting two point inside the ball, may go out and then back in. On the other hand the NTA condition implies that the two points can be connected with an arc of length comparable to the length of the radius of the ball. Precisely, there exists a constant $M>0$ such that the following property holds: 
\begin{center}{\it For every $x\in \partial\Omega$ and every $r>0$, there is exactly\\ one connected component of $B_r(x)\cap \Om$ that intersects $B_{r/M}(x)\cap \Om$.}
\end{center}	
\end{oss}

\begin{lemma}\label{reiflatprop}
Let $U$ be a solution of~\eqref{vectfb} and $x_0\in  \reg$. Then $\Omega_U$ is Reifenberg flat and NTA in a neighborhood of $x_0$. 
\end{lemma}
\begin{proof}
The proof follows by the same contradiction argument as in \cite[Proposition 5.9]{mtv}. Indeed, suppose that $\reg\ni x_n\to x_0$ and $r_n\to 0$ be such that $\partial \Omega_U$ is NOT $(\delta,r_n)$ flat in $B_{r_n}(x_n)$. Let $U_n:=U_{2r_n,x_n}$. Up to a subsequence $U_n$ converges to $U_0\in H^1(B_1;\R^k)$ which is a solution of \eqref{vectfb} in $B_1$. We will prove that $U_0$ is of the form \eqref{e:one_phase_bw}, then the conclusion will follow by the Hausdorff convergence of $\partial\Omega_{U_n}$ to $\partial\Omega_{U_0}$. Now, for fixed $0<r<1$ we have $W(U_n,0,r)=W(U,x_n,rr_n)\to W(U_0,x,r)$ as $n\to\infty$. Let now $\eps>0$ be fixed. Since $x_0\in\reg$, there is some $R>0$ such that $W(U,x_0,R)-\frac{\Lambda \omega_d}{2}\le \eps/2$. By the continuity of $W$ in $x$ we get that for $n$ large enough, $W(U,x_n,R)-\frac{\Lambda \omega_d}{2}\le \eps$ and, by the monotonicity of $W$, $W(U,x_n,rr_n)-\frac{\Lambda \omega_d}{2}\le \eps$. Passing to the limit in $n$ we obtain $W(U_0,x,r)-\frac{\Lambda \omega_d}{2}\le\eps$. Since $\eps$ is arbitrary, we get $W(U_0,x,r)=\frac{\Lambda \omega_d}{2}$. Finally, Lemma \ref{mono_weiss} implies that $U_0$ is one-homogeneous and $|B_1\cap\Omega_{U_0}|=\frac{\omega_d}2$. Thus, $U_0$ is necessarily of the form \eqref{e:one_phase_bw}, which concludes the proof.
\end{proof}
\subsection{Existence of a constant sign component}\label{sub:main}
After showing in the previous Section that the regular part of the free boundary is an NTA domain, we aim now to apply a boundary Harnack principle on it.
It was proved in \cite{jk} that in any NTA domain $\Omega\subset\R^d$ the Boundary Harnack Principle does hold, that is, if $u$ and $v$ are positive harmonic functions in $\Omega$, vanishing on the boundary $\partial\Omega\cap B_r$, then  
\[
\frac{v}{u}\ \mbox{ is H\"older continuous on }\ \overline{\Om}\cap B_r.
\]
The precise statement of the boundary Harnack property for harmonic functions which we will use in Lemma~\ref{costsign} is the following~\cite[Theorem~5.1 and Theorem~7.9]{jk}.
\begin{teo}[Boundary Harnack Principle for NTA Domains]\label{bdryharnack}
Let $\Omega\subset\R^d$ be an NTA domain and $A\subset \R^d$ an open set. For any compact $K\subset A$ there exists a constant $C>0$ such that for all positive harmonic functions $u,v$ vanishing continuously on $\partial \Om\cap A$, we have 
\begin{equation*}
C^{-1}\frac{v(y)}{u(y)}\leq \frac{v(x)}{u(x)}\leq C\frac{v(y)}{u(y)}\ ,\qquad\text{for all}\qquad x,y\in K\cap \overline \Om.
\end{equation*}
Moreover, there exists $\beta>0$, depending only on the NTA constants, such that the function $v/u$ is H\"older continuous of order $\beta$ in $K\cap \overline \Om$. In particular, for any $y\in \partial \Om \cap K$, the limit $\ds\lim_{\tiny\begin{array}{c} x\rightarrow y\\ \tiny x\in\Omega\end{array}}\frac{v(x)}{u(x)}$ exists. 
%
%
\end{teo}

\begin{oss}[Boundary Harnack principle for sign-changing $v$]
Theorem \ref{bdryharnack} still holds in the case when $u>0$ on the NTA domain $\Omega$ and $v$ is a harmonic function on $\Omega$ that may change sign. Indeed, if $v:B_1\cap\Omega\to\R$ is a harmonic function that changes sign in  $B_1\cap\Omega$ and vanishes on $\partial\Omega\cap B_1$, then we consider the harmonic extensions $h_+$ and $h_-$ solutions of the positive and negative parts of $v$:
$$\Delta h_{\pm}=0\quad\text{in}\quad\Omega\cap B_1,\qquad h_{\pm}=0\quad\text{on}\quad \partial\Omega\cap B_1,\qquad h_{\pm}=v_{\pm}\quad\text{on}\quad \partial B_1\cap\Omega.$$
Now, by Remark \ref{rem:ntapalleconn}, each of the functions $h_{\pm}$ is strictly positive or vanishes identically in $\Omega\cap B_{1/M}$. Thus, the claim follows by the boundary Harnack principle for positive functions applied to $h_+$ and $u$ (and $h_-$ and $u$), the fact that $v=h_+-h_-$ and a standard covering argument. 
\end{oss}

\begin{oss}
The constants $C$ and $\beta$ in the boundary Harnack principle do not change under blow-up. That is, given $x_0=0\in\partial\Omega$, there is $r_0>0$ such that for all harmonic functions $u,v$, solutions of 
$$\Delta u=\Delta v=0\quad\text{in}\quad \Omega_r\cap B_1\ ,\qquad u=v=0\quad\text{on}\quad \partial\Omega_r\cap B_1\ ,\qquad\Omega_r:=\frac1r\Omega\ ,\qquad 0<r<r_0\ ,$$
we have 
\begin{equation}\label{e:BH}
C^{-1}\frac{v(y)}{u(y)}\leq \frac{v(x)}{u(x)}\leq C\frac{v(y)}{u(y)}\ ,\qquad\text{for all}\qquad x,y\in B_{\sfrac12}\cap \overline \Om_r.
\end{equation}
\end{oss}
Following \cite{mtv} we aim to apply the boundary Harnack principle to the components of the vector $U$ in order to obtain that, for some $i\in\{1,\dots,k\}$, $|\nabla u_i|$ is Hölder continuous on $\partial\Omega_U$ and to apply the known regularity results for the one-phase Bernoulli problem to deduce that $\partial\Omega_U$ is $C^{1,\alpha}$. In our setting the functions $u_i$, $i=1,\dots,k$, may change sign, which is a major obstruction since \eqref{e:BH} can be applied only in the case when the denominator $u$ is strictly positive. In order to overcome this issue, we first show that, at every point $x_0$ of the regular free boundary $\reg$, there is a neighborhood of $x_0$ and a component $u_i$ which has constant sign in it.

\begin{lemma}\label{costsign}
Let $U=(u_1,\dots, u_k)$ be a solution for~\eqref{vectfb}. For all $x_0\in \reg$, there is $r>0$ and $i\in\{1,\dots,k\}$  such that the component $u_i$ has constant sign in $B_r(x_0)\cap\Omega_U$.
Moreover, there is a constant $C_{sign}>0$ such that $C_{sign}u_i\geq |U|$ in $B_r(x_0)\cap \Omega_U$. 
\end{lemma}
\begin{proof}
Without loss of generality $x_0=0$. Let $U_0\in \mathcal {BU}_U(x_0)$ and $U_n:=U_{r_n}$ be a blow-up sequence converging to $U_0$. By Lemma \ref{l:bw} there is a vector $\xi=(\xi_1,\dots,\xi_k)\in\R^k$ such that $|\xi|=\sqrt\Lambda$ and $U_0(x)=\xi x_d^+$ up to a rotation of $\R^d$. Now since $|\xi|=\sqrt{\Lambda}$, there is at least one component $\xi_i$ such that $|\xi_i|\ge\sqrt{\Lambda/k}$. Without loss of generality we can assume that $i=1$ and $\xi_1\ge\sqrt{\Lambda/k}$.

Let $\Omega_n=\Omega_{U_n}$ and $U_n=(u_{n1},\dots,u_{nk}):\Omega_{n}\cap B_1\to\R^k$; $u_{n1}^+$ and $u_{n1}^-$ be the positive and the negative parts of $u_{n1}$; $\widetilde u_{n}^+$ and $\widetilde u_{n}^-$ be the solutions of 
$$\Delta \widetilde u_{n}^{\pm} =0\quad\text{in}\quad \Omega_{n}\cap B_1,\qquad  \widetilde u_{n1}^{\pm}=0 \quad\text{on}\quad \partial\Omega_{n}\cap B_1,\qquad  \widetilde u_{n}^{\pm}=u_{n1}^{\pm} \quad\text{on}\quad \Omega_{n}\cap\partial B_1.$$ 
Now, notice that both  $u_{n1}^+$ and $u_{n1}^-$ are subharmonic on $\Omega_{n}\cap B_1$. Thus,  
$$\widetilde u_n^+-\widetilde u_n^-=u_{n1}^+-u_{n1}^-=u_{n1}\ ,\quad
\widetilde u_n^{+}\geq u_{n1}^{+},\quad\text{and}\quad
\widetilde u_n^{-}\geq u_{n1}^{-}\quad \mbox{in}\quad\Omega_n\cap B_1.$$
Let $M$ be the constant from Remark \ref{rem:ntapalleconn}. By the fact that the blow-up limit $U_0$ has a positive first component, for a fixed $n$, in the ball $B_{\sfrac1{M}}$ can happen exactly one of the following situations: 
$$(i)\quad \widetilde u_{n}^+>0\quad\text{and}\quad \widetilde u_n^->0\quad\text{in}\quad\Omega_n\cap B_{\sfrac1{M}}\;;\qquad (ii)\quad \widetilde u_n^+>0\quad\text{and}\quad \widetilde u_n^-=0\quad\text{in}\quad\Omega_n\cap B_{1/M}.$$
Moreover, again by Remark \ref{rem:ntapalleconn} we obtain that in both cases we have that $\Omega_n\cap B_{\sfrac1{M}}=\{\widetilde u_n^+>0\}\cap B_{\sfrac1{M}}$, while if $(i)$ holds, then also $\Omega_n\cap B_{\sfrac1{M}}=\{\widetilde u_n^->0\}\cap B_{\sfrac1{M}}$.
Now, notice that in the case $(ii)$ the first part of the claim of the Lemma is trivial, so we concentrate our attention at the case $(i)$. Let $\ds x_M:=\frac{e_d}{2M}$ and $\ds r_M:=\frac1{4M}$.  Recall that $U_{n}$ converges uniformly to $U_0$ and $\partial \Omega_{n}$ converges to $\partial \Omega_{U_{0}}=\{x_d=0\}$ in the Hausdorff distance. Then, for every $\eps>0$, there is $n_0>0$ such that for every $n\ge n_0$ we have 
$$B_{r_M}(x_M)\subset \Omega_n\ ,\quad u_{n1}^+(x_M)\ge \sqrt{\frac{\Lambda}{k}}\frac{r_M}2\ ,\quad\text{and}\quad |u_{n1}^-|\le \eps\quad\text{in}\quad B_1.$$
Now, by the definition of $\widetilde u_n^+$ and $\widetilde u_n^-$ and the maximum principle (applied to $\widetilde u_n^-$), we have 
$$\widetilde u_n^+(x_M)\ge \sqrt{\frac{\Lambda}{k}}\frac{r_M}2\qquad\text{and}\qquad \widetilde u_n^-(x_M)\le \eps.$$
Finally, by \eqref{e:BH}, we obtain 
$$\frac{\widetilde u_n^-(x)}{\widetilde u_n^+(x)}\le C\frac{\widetilde u_n^-(x_M)}{\widetilde u_n^+(x_M)}\le \eps C\sqrt{\frac{\Lambda}{k}}\frac{r_M}2 \qquad\text{for every}\qquad x\in \Omega_n\cap B_{\sfrac12}.$$  
Choosing $\eps$ such that the right-hand side is smaller than one, we get  
$$u_{n1}(x)=\widetilde u_n^+(x)-\widetilde u_n^-(x)>0\qquad\text{for every}\qquad x\in \Omega_n\cap B_{\sfrac12},$$  
which proves the first claim. The second part of the statement follows by the boundary Harnack principle applied to $u_{n1}$ and every component $u_{ni}$, for $i=2,\dots,k$. 
\end{proof}

\subsection{The regular part of the free boundary is $C^{1,\alpha}$}\label{sub:C1alpha}

In the following lemma we show that the positive optimal component is locally a solution of a one-phase scalar free boundary problem with Hölder condition on the free boundary. The $C^{1,\alpha}$ regularity of $\reg$ then follows by known results on the regularity of the one-phase free boundaries (see \cite[Theorem 1.1]{desilva}). 

\begin{lemma}\label{locoptu1}
Let $U=(u_1,\dots,u_k)$ be a minimizer for~\eqref{vectfb} and $0\in \reg$ and let the first component be of constant sign in a neighborhood of $0$, that is $u_1>0$ in $B_{r_0}\cap \Omega_U$. Then there is a constant $0<c_0\le 1$, $0<r\leq r_0$ and a H\"older continuous function $g:B_{r}\cap\partial\Omega_U\to [c_0,1]$ such that $u_1$ is a viscosity solution to the problem
\begin{equation*}
-\Delta u_1=0 \quad\mbox{in}\quad\Omega_U\cap B_r\, \qquad u_1=0\quad \mbox{on}\quad\partial \Omega_U\cap B_r\ ,\qquad |\nabla u_1|=g\sqrt\Lambda \quad \mbox{on}\quad \partial\Omega_U\cap B_r.
\end{equation*}
\end{lemma}
\begin{proof}
First notice that, by Lemma \ref{reiflatprop}, $\Omega_U$ is an NTA domain in a neighborhood of $0$ and there exists $\beta>0$, depending only on the NTA constants, such that for $i=2,\dots, k$,  
$\sfrac{\ds u_i}{\ds u_1}$ is H\"older continuous of order $\beta$ on $\overline{\Om}_U\cap B_r$,
for some $r\leq r_0.$
In particular, for every $x_0\in \Om^{(\sfrac12)}\cap B_r$, the limit $\ \ds g_i(x_0):=\lim_{\Omega\ni x\to x_0}\frac{u_i(x)}{u_1(x)},\ $
exists and $g_i:B_r\cap\partial\Omega\to\R$ is an $\beta$-H\"older continuous function.
Then we have 
$$u_i=g_i u_1\quad\text{on}\quad B_r\cap\overline\Omega\qquad\text{and}\qquad u_1=g|U|\quad\text{on}\quad B_r\cap\overline\Omega,\quad\text{where}\quad g:=\left(1+g_2^2+\dots+g_k^2\right)^{-\sfrac12}.$$
We notice that $g$ is a $\beta$-H\"older continuous function on $\overline\Omega\cap B_r$ for some $\beta>0$ and is such that $c_0\le g\le 1$, where $c_0=1/C_{sign}$ and $C_{sign}$ is the constant from Lemma~\ref{costsign}. Suppose now that the function $\varphi\in C^1(\R^d)$ is touching $u_1$ from below (see Definition \ref{viscoptimality}, note that it is local) in a point $x_0\in\partial\Omega\cap B_r$. For $\rho$ small enough, there is a constant $C>0$ such that  
$$\frac1{g(x)}\ge \frac1{g(x_0)}-C|x-x_0|^\gamma\ge 0\quad\text{for every}\quad x\in \overline\Omega\cap B_{\rho}(x_0),$$
and so, setting $\psi(x)=\varphi(x)\big(\frac1{g(x_0)}-C|x-x_0|^\gamma\big)$, we get that $\psi(x_0)=|U|(x_0)$ and 
$$\psi(x)\le u_1(x)\left(\frac1{g(x_0)}-C|x-x_0|^\gamma\right)\le |U|(x)\quad\text{for every}\quad x\in \overline\Omega\cap B_{\rho}(x_0),$$
that is in the ball $B_{\rho}(x_0)$ we have that $\psi$ touches $|U|$ from below in $x_0$. On the other hand, $\psi$ is differentiable in $x_0$ and $|\nabla \psi(x_0)|=\frac1{g(x_0)}|\nabla \varphi(x_0)|$. Since $U$ is a viscosity solution of \eqref{viscsolUtion} we obtain that 
$$\sqrt\Lambda\ge |\nabla \psi(x_0)|=\frac1{g(x_0)}|\nabla \varphi(x_0)|,$$
which gives the claim, the case when $\varphi$ touches $u_1$ from below being analogous.
\end{proof}

\subsection{Higher regularity. The regular part of the free boundary is $C^\infty$}\label{sub:Cinfty}
Thanks to Lemma~\ref{costsign}, we can apply the arguments of~\cite[Section~5]{mtv} in order to obtain the $C^\infty$ regularity of $\reg$, using the component of locally constant sign provided by Lemma~\ref{costsign} instead of $u_1$ in the boundary Harnack principle~\cite[Lemma~5.12]{mtv}. We recall here the updated statements for the reader's sake.

In order to pass from $C^{1,\alpha}$ to $C^{\infty}$ we need an improved boundary Harnack principle, as it was proved by De Silva and Savin~\cite{dss} for harmonic functions. 

\begin{teo}[Improved boundary Harnack principle]\label{imprbh}
Let $U=(u_1,\dots,u_k)$ be a minimizer for~\eqref{vectfb}, $0\in \reg$ and let the first component be of constant sign in a neighborhood of $0$, that is $u_1>0$ in $B_{r_0}\cap \Omega_U$.
There exists $R_0<\sfrac12$ such that, if for $r<\min{\{R_0,r_0\}}$, $\reg\cap B_r$ is of class $C^{k,\alpha}$ for $k\geq 1$, then for all $i=2,\dots, k$ we have 
\[
\frac{u_i}{u_1}\ \mbox{is of class }C^{k,\alpha} \mbox{ on }\overline{\Om_U}\cap B_r.
\] 
In particular, for every $x_0\in \reg\cap B_r$, the limit $\ \ds g_i(x_0):=\lim_{\Omega_U\ni x\to x_0}\frac{u_i(x)}{u_1(x)},\ $
exists and $g_i:B_r\cap\partial\Omega_U\to\R$ is a $C^{k,\alpha}$ function. 
\end{teo}
\begin{proof}
In order to show the claim, it is enough to apply~\cite[Theorem~2.4]{dss} for the case $k=1$ and~\cite[Theorem~3.1]{dss} for the case $k\geq 2$. 
\end{proof}

At this point we are in position to prove the full regularity of $\reg$.

\begin{lemma}\label{regularityfinalinfty}
Let $U=(u_1,\dots,u_k)$ be a minimizer for~\eqref{vectfb}, $0\in \reg$ and let the first component be of constant sign in a neighborhood of $0$, that is $u_1>0$ in $B_{r_0}\cap \Omega_U$. Then $\reg$ is locally a graph of a $C^{\infty}$ function.
\end{lemma}
\begin{proof}
The smoothness of the free boundary follows by a bootstrap argument as in \cite{kn}. 
Let us assume that $\reg$ is locally $C^{k,\alpha}$ regular for some $k\geq 1$, the case $k=1$ being true thanks to Section~\ref{sub:C1alpha}. We will prove that $\reg$ is locally $C^{k+1,\alpha}$. By Lemma~\ref{locoptu1} the first component $u_1$ is locally a (classical) solution to the problem
\begin{equation*}
\Delta u_1=0\quad\mbox{in}\quad\Omega_U\ , \qquad u_1=0\quad \mbox{on}\quad  \reg\ ,\qquad |\nabla u_1|=g\sqrt\Lambda \quad \mbox{on}\quad \reg.
\end{equation*}
Now thanks to Lemma~\ref{imprbh} and the definition of $g$ we have that $g$ is a $C^{k,\alpha}$ function. Now by \cite[Theorem 2]{kn} we have that $\reg$ is locally a graph of a $C^{k+1,\alpha}$ function, and this concludes the proof.
\end{proof}


\section{Structure of the branching free boundary}\label{s:branch}
In this section we study in more detail the set of branching points $\2$. By the results of Subsection \ref{sub:BUclassification} we know that for a $x_0\in \partial\Omega_U$ we have
$$x_0\in \2\ \Leftrightarrow\  x_0\in\Omega_U^{(1)}\ \Leftrightarrow\  \text{Every blow-up $U_0\in\mathcal{BU}_U(x_0)$ is a linear function.}$$
In the following Subsection  \ref{sub:rank} we prove that the rank of $U_0$ depends only on $x_0$. Then, in Subsection \ref{sub:dim} we stratify the singular set according to the rank at each point and finally, in the last subsection, we give some measure theoretical criterion for the uniqueness of the blow-up.
\subsection{Definition of the strata and decomposition of $\2$}\label{sub:rank}
\begin{lemma}\label{l:QU}
Let $U=(u_1,\dots,u_k)$ be a solution of \eqref{vectfb} and $Q\in \mathcal{O}(k)$ be an orthogonal matrix. Then $V:=QU$ is also a solution of \eqref{vectfb} corresponding to the boundary datum $Q\Phi$.
\end{lemma}
\begin{proof}
It is sufficient to notice that for every $U:D\to\R^k$ we have $|QU|=|U|$ and $|\nabla (QU)|^2=|\nabla U|^2$.
\end{proof}
\begin{lemma}\label{l:rank}
Let $U=(u_1,\dots,u_k)$ be a solution of \eqref{vectfb} and $x_0\in\Omega_U^{(1)}\cap\partial\Omega_U$. Then every blow-up $U_0\in \mathcal{BU}_U(x_0)$ is a linear function given by a matrix $A\in M^{d\times k}(\R)$, whose rank does not depend on $U_0$ but only on $x_0$ and $U$. 
\end{lemma}
\begin{proof}
Without loss of generality we may assume that $x_0=0$.
Let $U_0\in \mathcal{BU}_U(0)$, $U_0(x)=Ax$, be a blow-up such that $\text{rank}\,A=j$ for some $j\in\{1,\dots,k\}$. We claim that all the blow-ups in $\mathcal{BU}_U(x_0)$ are of rank $j$.

We first prove the claim in the case $j=1$. Indeed, consider a matrix $Q\in\mathcal O(k)$ such that $QAx=(\nu\cdot x,0,\dots,0)$ for some $\nu\in\R^d$ and consider the vector valued function $V=(v_1,\dots,v_k):=QU$, which is also a solution \eqref{vectfb} by Lemma \ref{l:QU}. Now, since each of the components $v_i$ is a harmonic function on the set $\{v_i\neq0\}$, the Alt-Caffarelli-Friedman monotonicity formula (see \cite{acf}) gives that the function
\begin{equation}\label{acffun}
r\mapsto\Phi(r,v_i):=\left(\frac{1}{r^2}\int_{B_r}\frac{|\nabla v_i^+|^2}{|x|^{d-2}}\,dx\right)\left(\frac{1}{r^2}\int_{B_r}\frac{|\nabla v_i^-|^2}{|x|^{d-2}}\,dx\right)=\int_{B_1}\frac{|\nabla (v_i)_r^+|^2}{|x|^{d-2}}\,dx\int_{B_1}\frac{|\nabla (v_i)_r^-|^2}{|x|^{d-2}}\,dx,
\end{equation}
is increasing in $r$, where as usual $(v_i)_r(x):=\frac1r v_i(rx)$. Now, since for $i\in\{2,\dots,k\}$ the $i^{th}$ component of the blow-up $QA\in \mathcal{BU}_V(0)$ constantly vanishes, we have that $\Phi(0,v_i):=\lim_{r\to 0}\Phi(r,v_i)=0$. In particular, the $i^{th}$ component of any blow-up $V_0\in \mathcal{BU}_V(0)$ should vanish and so, the only non-vanishing component of $V_0$ is the first one (recall that the blow-ups are non-trivial by the non-degeneracy of the solutions of \eqref{vectfb}). Now since $\mathcal{BU}_V(0)=Q(\mathcal{BU}_U(0))$ we obtain that the rank of any blow-up $\mathcal{BU}_U(0)$ is precisely one, which proves our claim. 

Let us now suppose that $2\le j\le k$ and that the claim holds for all $i\in\{1,\dots,j-1\}$. We will now prove the claim for $j$. Reasoning as above, we first find a matrix $Q\in\mathcal O(k)$ such that the last $k-j$ components of $QA$ vanish, that is $(QA)_{j+1}=\dots=(QA)_k=0\in\R^k$. Then, we consider the vector valued function $V=(v_1,\dots,v_k):=QU$ and notice that, for all $i=1,\dots,k$, the function $r\mapsto\Phi(r,v_i)$ is increasing in $r$. As above, the strong $H^1$ convergence of the blow-up sequences implies that $\Phi(0,v_{j+1})=\dots=\Phi(0,v_k)=0$ and that the components $j+1,...,k$ of any blow-up $V_0\in \mathcal{BU}_V(0)$ do vanish identically. Thus, the rank of $V_0$ is at most $j$. On the other hand, since the claim does hold for every $i\in\{1,\dots,j-1\}$, the rank of $V_0$ is precisely $j$, which concludes the proof. \end{proof}

Lemma \ref{l:rank} allows us to define, for every $j\in\{1,\dots,d\}$, the stratum 
\begin{equation}\label{e:S_j}
\mathcal S_j:=\Big\{x_0\in \Omega_U^{(1)}\cap \partial\Omega_U\ :\ \text{every blow-up}\ U_0\in \mathcal{BU}_U(x_0)\ \text{has rank}\ j\Big\}.
\end{equation} 
Again, by Lemma \ref{l:rank}, the singular set $\partial\Omega_U\cap \Omega_U^{(1)}$ can be decomposed as a disjoint union
\begin{equation}\label{e:sing2}
\Omega_U^{(1)}\cap \partial\Omega_U=\bigcup_{j=1}^d\mathcal S_j.
\end{equation} 


\subsection{Dimension of the strata}\label{sub:dim}

In this subsection we give an estimate on the Hausdorff dimension, $\text{dim}_{\mathcal H}$ of the stratum $\mathcal S_j$. The proof is based on a well-known technique in Geometric Measure Theory known as Federer Reduction Principle. 

Given $A\in\R^d$, $0\le s<\infty$ and $0<\delta\le\infty$, we recall the notations 
$$\mathcal H^s_\delta(A)=\frac{\omega_s}{2^s}\inf\Big\{\sum_{i=1}^\infty (\text{diam}\,C_i)^s\ :\ A\subset\bigcup_{i=1}^\infty C_i\ ,\ \text{diam}\,C_i<\delta\Big\}\ ,\qquad \mathcal H^s(A)=\sup_{\delta\ge 0} \mathcal H^s_\delta(A)\ ,$$
$$\text{dim}_{\mathcal H} A=\inf\{s\ge 0\ :\ \mathcal H^s(A)=0\}.$$
It is well known that $\mathcal H^s(A)=0$ if and only if $\mathcal H^s_\infty(A)=0$. The other fact (for a proof we refer to \cite[Proposition 11.3]{giusti}) that we will use is that 
\begin{equation}\label{e:giusti-federer}
\limsup_{r\to 0}\frac{\mathcal H^s_\infty(A\cap B_r(x))}{2^{-s}\omega_s r^s}\ge 1\quad\text{for $\mathcal H^s$ - almost every}\ x\in A.
\end{equation}

\begin{teo}\label{t:stratification}
Let $U:\R^d\supset D\rightarrow\R^k$ be a solution of \eqref{vectfb} and $\mathcal S_j$ be as in \eqref{e:S_j}. If $j=d$, then $\mathcal S_j$ is a discrete subset of $D$. More precisely each point of $\mathcal S_d$ is isolated in $\partial\Omega_U$. If $1\le j<d$, then $\mathcal S_j$ is a set of Hausdorff dimension $\dim_{\mathcal H}\mathcal S_j\le d-j$. 
\end{teo}
\begin{proof}
We start with the first claim. Suppose that $x_0\in \mathcal S_d$ and there is a sequence $\partial\Omega_U\ni x_n\to x_0$. Taking $r_n=|x_n-x_0|$, $\xi_n:=(x_n-x_0)/r_n$, $\xi_0=\lim_{n\to\infty}\xi_n$ and a blow-up limit $U_0\in\mathcal{BU}_U(x_0)$ of the sequence $U_{r_n,x_0}$ we obtain that $\xi_0\in\partial B_1$, $U_0(\xi_0)=0$ and so $\dim \text{Ker}\, U_0\ge 1$, which is a contradiction with the definition of $\mathcal S_d$. 

Let now $j<d$. Suppose by contradiction that there is $\eps>0$ and a solution $U$ of \eqref{vectfb} such that $\mathcal H^{d-j+\eps}(\mathcal S_j)>0$. Then, by \eqref{e:giusti-federer}, we get that there is a point $x_0\in \mathcal S_j$ such that 
\begin{equation}\label{e:choice_x_0}
\limsup_{r\to 0}\frac{\mathcal H^{d-j+\eps}_\infty(\partial\Omega_U\cap B_r(x_0))}{r^{d-j+\eps}}\ge \limsup_{r\to 0}\frac{\mathcal H^{d-j+\eps}_\infty(\mathcal S_j\cap B_r(x_0))}{r^{d-j+\eps}}\ge 2^{-(d-j+\eps)}\omega_{d-j+\eps}.
\end{equation}
Now let $r_n\to0$ be a sequence realizing the first limsup above and $U_n=U_{r_n,x_0}$ be a blow-up sequence converging to some $U_0\in\mathcal{BU}_U(x_0)$. In particular, $\partial\Omega_{U_n}$ converges in the Hausdorff distance to $\partial\Omega_{U_0}$. Now, since $\mathcal H^{s}_\infty$ is upper semi-continuous with respect to the Hausdorff convergence of sets, \eqref{e:choice_x_0} gives that 
$$\mathcal H^{d-j+\eps}_\infty\big(\partial\Omega_{U_{0}}\cap B_1\big)\ge \lim_{n\to \infty}\mathcal H^{d-j+\eps}_\infty\big(\partial\Omega_{U_{n}}\cap B_1\big)\ge 2^{-(d-j+\eps)}\omega_{d-j+\eps},$$
which is in contradiction with the fact that $\mathcal H^{d-j+\eps}\big(\partial\Omega_{U_{0}}\cap B_1\big)=0$. 
\end{proof}

\begin{oss}
A more refined argument in the spirit of Naber and Valtorta, essentially based on the Weiss' monotonicity formula and the structure of the blow-up limits, can be used to deduce that the set $\mathcal S_j$ has finite $(d-j)$-dimensional Hausdorff measure.  
For more details on this technique in the context of the free-boundary problems considered in this paper we refer the reader to \cite{ee}. 
\end{oss}

\subsection{A density criterion for the uniqueness of the blow-up limit}\label{sub:crit}
The uniqueness of the blow-up limit is a central question in free boundary problems and is strictly related to the $C^1$-rectifiability of the singular set. It remains a major open question even in the case of the two-phase problem corresponding to the case $k=1$. In this last subsection we give a general criterion for the uniqueness of the blow-up at the singular points, which depends only on the Lebesgue density of the positivity set $\Omega_U$ (see Proposition \ref{p:uniqueness}). Now, even if at this point this criterion by itself is not sufficient for the conclusion, it provides a proof of the fact that the regularity of the singular set only reduces to a control over the measure of the nodal set  $B_r\setminus\Omega_U$. We prove the lemma by choosing a power rate of convergence, but the argument can be carried out under more general assumptions. For example, a logarithmic decay of the density still translates into a decay of the Weiss energy. This, again implies a blow-up uniqueness and a logarithmic rate of convergence (see \cite{esv}). 
In this subsection we use the notations $W(U,r):=W(U,0,r)$ and $W_0(U,r):=W_0(U,0,r)$, where
$$W_0(U,x_0,r)=\frac{1}{r^d}\int_{B_r(x_0)}|\nabla U|^2\,dx-\frac{1}{r^{d+1}}\int_{\partial B_r(x_0)}|U|^2\,d\HH^{d-1}.$$
\begin{prop}\label{p:uniqueness}
Suppose that $U$ is a solution of \eqref{vectfb} and $x_0\in\partial\Omega_U$. If there are constants $C>0$ and $\alpha>0$ such that 
$$\frac{|B_r(x_0)\setminus\Omega_U|}{r^d}\le Cr^{\alpha}\quad\text{for every}\quad 0<r<\text{dist}\,(x_0,\partial D),$$
then there is a unique blow-up $U_0\in\mathcal{BU}_U(x_0)$ and we have the estimate $\|U_{r,x_0}-U_0\|_{L^2(\partial B_1)}\le Cr^\beta$ for some $\beta=\beta(\alpha,d)$.
\end{prop}
\begin{proof}
Let $x_0=0$ and $r>0$ be fixed. Let $H:B_r\to\R^k$ be the harmonic extension of $U$ in the ball $B_r$. A classical estimate for harmonic functions (see \cite[Lemma 2.5]{sv}) states that there is a dimensional constant $\bar\eps>0$ such that 
\begin{equation}\label{e:harmonic}
(1+\eps)W_0(H,r)\le W_0(Z,r)\quad\text{for every}\quad \eps\in[0,\bar\eps],
\end{equation}
where $Z$ is the one-homogeneous extension of $U$ in the ball $B_r(x_0)$.
On the other hand, $|B_r\setminus\Omega_H|=0$ and so, the optimality of $U$ gives 
\begin{equation}\label{e:UvsH}
W_0(U,r)\le W_0(H,r)+r^{-d}\Lambda |B_r\setminus\Omega_U|.
\end{equation}
Finally, we notice that for every function $U$ we have the formula 
\begin{equation}\label{e:derivataW0}
\frac{\partial}{\partial r}W_0(U,r)=\frac{d}{r}\left(W_0(Z,r)-W_0(U,r)\right)+\frac{1}{r^{d+2}}\sum_{i=1}^k\int_{\partial B_r} |x\cdot\nabla u_i-u_i|^2\,d\HH^{d-1}.
\end{equation}
Now, using \eqref{e:derivataW0}, \eqref{e:harmonic} and \eqref{e:UvsH}, we have  
\begin{align*}
\frac{\partial}{\partial r}W_0(U,r)&\ge \frac{d}{r}\left(W_0(Z,r)-W_0(U,r)\right)\ge \frac{d}{r}\left(W_0(Z,r)-W_0(H,r)-r^{-d}\Lambda |B_r\setminus\Omega_U|\right)\\
&\ge \frac{d}{r}\left(\eps W_0(H,r)-r^{-d}\Lambda |B_r\setminus\Omega_U|\right)\ge \frac{d}{r}\left(\eps W_0(U,r)-(1+\eps)r^{-d}\Lambda |B_r\setminus\Omega_U|\right)\\
&\ge \frac{d\eps}{r} W_0(U,r)-2d\Lambda C r^{\alpha-1}.
\end{align*}
In particular, this implies that the function 
$$r\mapsto \frac{W_0(U,r)}{r^{\eps d}}+\frac{2d\Lambda C}{\alpha-d\eps} r^{\alpha-d\eps}$$
is increasing in $r$ and so, choosing $\eps=\frac{\alpha}{2d}$, we get that there is a constant $C_{U,x_0}$ depending on $U$ and the point $x_0=0\in D$ such that 
$$W_0(U,r)\le C_{U,x_0} r^{\alpha/2}\qquad\text{and}\qquad W(U,r)-\Lambda\omega_d=W_0(U,r)-\Lambda \frac{|B_r\setminus\Omega_U|}{r^d}\le W_0(U,r)\le C_{U,x_0} r^{\alpha/2}.$$
Now, the uniqueness of the blow-up and the convergence rate follow by a standard argument (see \cite{sv}).  
\end{proof}

\end{document}